\documentclass{amsart}[12]
\usepackage{amsmath, amsrefs}
\usepackage{graphicx}
\usepackage{latexsym}
\usepackage{color}
\usepackage{amscd}
\usepackage[all]{xy}
\usepackage{enumerate}
\usepackage{pinlabel}
\usepackage{hyperref}
\newtheorem{thm}{Theorem}

\newtheorem{cor}[thm]{Corollary}
\newtheorem{prop}[thm]{Proposition}
\newtheorem*{claim}{Claim}

\newtheorem{theorem}[thm]{Theorem}

\theoremstyle{definition}

\newtheorem{remark}[thm]{Remark}

\newcommand{\spinc}{\text{spin}^c}

\begin{document}

\title{0-Concordance of 2-knots.}

\author[N. Sunukjian]{Nathan Sunukjian}
\address{Department of Mathematics and Statistics, Calvin University, Grand Rapids, MI 49546}
\email{nss9@calvin.edu}

\begin{abstract}
In this paper we investigate the 0-concordance classes of 2-knots in $S^4$, an equivalence relation that is related to understanding smooth structures on 4-manifolds. Using Rochlin's invariant, and invariants arising from Heegaard-Floer homology, we will prove that there are infinitely many 0-concordance classes of 2-knots. 
\end{abstract}

\maketitle

\section{Introduction}

A 2-knots is a smooth embedding of $S^2$ in $S^4$, and a \emph{0-concordance} of 2-knots is a concordance with the property that every regular level set of the concordance is just a collection of $S^2$'s. In his thesis, Paul Melvin proved that if two 2-knots are 0-concordant, then a Gluck twist along one will result in the same smooth 4-manifold as a Gluck twist on the other. He asked the following question: Are all 2-knots 0-slice (i.e. 0-concordant to the unknot)? In \cite{N}, we generalized Melvin's theorem to surgeries on higher genus surfaces in arbitrary 4-manifolds, and proved that there are infinitely many 0-concordance classes of higher genus surfaces.


In this paper, we will prove the following theorem, which answers Melvin's question and question 1.105a on Kirby's problem list.

\begin{theorem}\label{mainthm}
There are an infinite number of 0-concordance classes of 2-knots.
\end{theorem}

It remains an open question whether Gluck twists can be used to construct exotic 4-manifolds (in particular exotic $S^4$'s), and our theorem shows that one cannot hope to answer it by showing that all 2-knots are 0-concordant to the unknot. 

Every 2-knot is slice \cite{Ke}, but this theorem shows that not every 2-knot is 0-slice, which parallels the situation for knots in $S^3$. In fact, the theory of 0-concordance of 2-knots parallels that of concordance in $S^3$ in a few other ways as well. For example, similar to the fact that a slice knot has vanishing signature, we will show:

\begin{theorem}
If the 2-twist spin of a quasi-alternating knot $K$ is 0-slice, then the signature of $K$ vanishes. 
\end{theorem}



How can 0-concordance classes be distinguished? Quandle cohomology has been useful for studying ribbon concordance of 2-knots (see. eg. \cite{CSS}), however, as we will see, what works to distinguish ribbon concordance does not necessarily apply to 0-concordance. In this paper we'll show that the Rochlin invariant of a knot (defined in e.g. \cite{R}) can distinguish 16 different 0-concordance classes (Section 3), and Heegaard-Floer correction terms (specifically the twisted $d$-invariants defined by Behrens-Golla in \cite{BG}) can be used to distinguish infinitely many (Sections 4 and 5). 







\section{Basics about 2-knots and 0-concordance.}

There are two families of 2-knots that are relatively easy to describe: ribbon knots, and twist spun knots. Our study of 2-knots will be based on invariants derived from Seifert hypersurfaces of 2-knots (i.e. 3-manifolds in $S^4$ that have the knot as their boundary), and for both twist spun knots and ribbon knots it is easy to describe natural Seifert hypersurfaces. We will describe these hypersurfaces in this section as well as explain how the Seifert hypersurfaces of 2-knots which are 0-concordant are related.

\subsection{Spun knots}

For a knot $K= S^1 \subset S^3$, one can define the spun 2-knot $S_0(K) \subset S^4$ as indicated in Figure \ref{twist}: The knot $K$ gives rise to an arc $\kappa$ in the 3-dimensional upper half plane $R^3_+$, and thinking of $R^3_+$ embedded in $R^4$, we can spin $R^3_+$ around a central axis sweeping out all of $R^4$, and the arc sweeps out the spun knot. Compactifying $R^4$ then gives us a knot in $S^4$. The \emph{n-twist spun knot}, here denoted $S_n(K)$,  is defined similarly except we rotate $\kappa$ in three-space $n$-times as we sweep it out through $R^4$. Precise coordinate definitions can be found in Zeeman, \cite{Z}, where this construction originally appeared. Note that although Zeeman carefully defines which direction the ``spinning'' should be done in, in an important sense he does not distinguish between the $n$-twist spun knot, and the $(-n)$-twist spun knot. Although at face value $S_n(K)$ and $S_{-n}(K)$ are constructed by spinning is in opposite directions, there is an automorphism of of $S^4$ taking one of these knots to the other. On the other hand, this automorphism does not preserve an orientation on the knot.\footnote{The invariants we look at in this paper can be used to show that often these knots differ as oriented knots. See Remark \ref{sym}.} 
\begin{figure}
\labellist
\small\hair 2pt
\pinlabel {$\times n$} at 42 92
\endlabellist	
\includegraphics{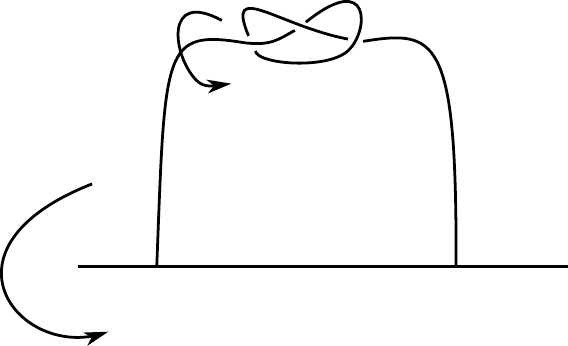} \label{twist}
	\caption{The n-twist spin of a knot.}
	
	\end{figure}

For the purposes of this paper, the most important result about twist spun knots is the following from Zeeman:

\begin{prop}[\cite{Z}]\label{zeeman}
For $n\geq 2$, the complement of a neighborhood of an $n$-twist spun knot fibers over $S^1$, where the fiber is the punctured n-fold branched cover of $S^3$ over the knot $K$, and the monodromy is given by the branching action.\footnote{The 0-twist spun knots, i.e. spun knots, are all ribbon, and have their Seifert hypersurfaces described in the next section. The 1-twist spin of any knot is always unknotted.} 
\end{prop}

This gives us a method for finding Seifert hypersurfaces: The fiber of a twist spun knot is a natural Seifert hypersurfaces for the knot.

\subsection{Ribbon knots and ribbon concordance}
Ribbon 2-knots are described as follows: begin with a collection of $n$ unknotted $S^2$'s in $S^4$, and add $n-1$ tubes connecting them in such a way as to get a connected surface. We say that there is a \emph{ribbon concordance} from $K_1$ to $K_2$ if we can add unknotted $S^2$'s to $K_1$ followed by a series of tubes to arrive at $K_2$. Alternatively, we say they are ribbon concordant if there is a concordance where the critical level sets all have index 0 or 1. We showed in \cite{N} that a 0-concordance is really just the composition of two ribbon concordances. Specifically, if $K_1$ is 0-concordant to $K_2$, then there is a third knot that is ribbon concordant to both.

For our purposes, the most important property of ribbon 2-knots is the fact that a ribbon 2-knot has $(\sharp_{(n-1)} S^1\times S^2)^{\circ}$ as a Seifert hypersurface. This generalizes as follows.


\begin{prop}
If 2-knot $K_1$ is ribbon concordant to $K_2$, and $M^\circ$ is a Seifert hypersurface for $K_1$, then $M^{\circ} \# n(S^1\times S^2)$ is a Seifert hypersurface of $K_2$ for some $n\geq 0$. 
\end{prop}

The construction of these Seifert hypersurfaces will mirror the construction of Seifert hypersurfaces for ribbon knots (see e.g. \cite{Y} or \cite{R}), so we will only sketch it here.

\begin{proof}
Suppose $K_2$ is constructed as $K_1$ plus a disjoint union of $S^2$'s, plus a series of tubes attached along arcs $\gamma_i$. Notice that $K_1$ and the $S^2$'s bound $M^{\circ}$ plus some $D^3$'s embedded in $S^4$ (see Figure \ref{ribpic}). Call this disconnected manifold $M'$. The $\gamma_i$ intersect $M'$ in isolated points, and by isotoping these arcs, we can assume that these intersections pair up, positive with negative. By connect summing from a positive intersection to a negative along $\gamma_i$, we can replace $M'$ by $M''$, the disjoint union of $M^{\circ} \#_m S^1\times S^2$ with possibly several copies of $(\#_{m_j}S^1\times S^2)^{\circ}$. Finally, by adding 1-handles along the $\gamma_i$, we boundary connect sum the components of $M''$ together, and the boundary will be $K_2$. This gives the desired Seifert hypersurface.

\end{proof}

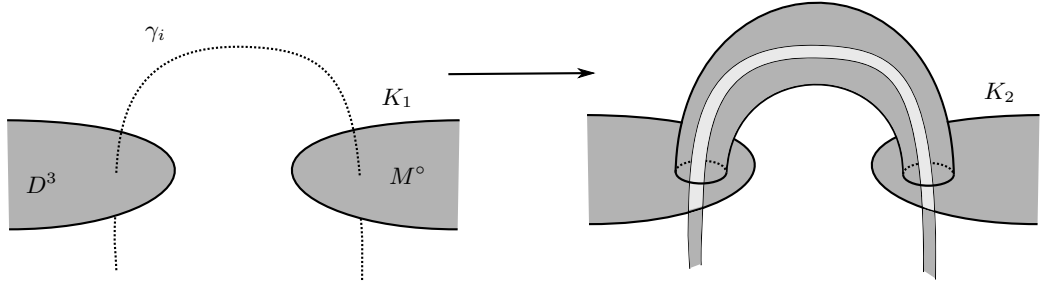
\begin{figure}
\labellist
\small\hair 2pt
\pinlabel {$D^3$} at 13 37
\pinlabel {$M^{\circ}$} at 151 40
\pinlabel {$\gamma_i$} at 56 93
\pinlabel {$K_1$} at 147 69
\pinlabel {$K_2$} at 374 71
\endlabellist	
\includegraphics{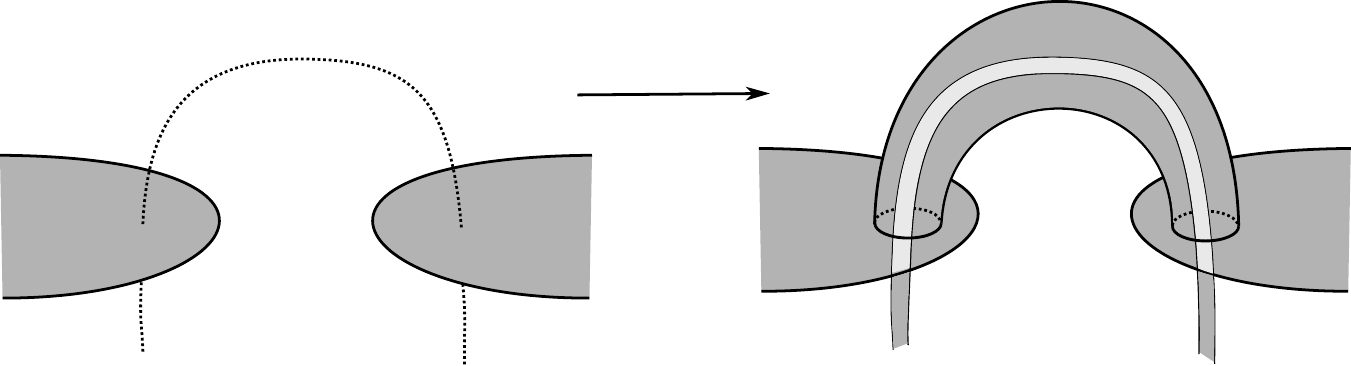}\label{ribpic}
	\caption{Seifert hypersurfaces of ribbon concordances.}
	
	\end{figure}

\begin{cor}\label{seif0}
If 2-knots $K_1$ and $K_2$ are 0-concordant, with Seifert hypersurfaces $M_1^{\circ}$ and $M_2^{\circ}$, then there is a 2-knot $K$ that has both $M_1^{\circ} \# n(S^1\times S^2)$ and $M_2^{\circ} \# m(S^1\times S^2)$ as Seifert hypersurfaces for some $n,m\geq 0$.
\end{cor}
\begin{proof}
This follows from the fact \cite{N} that a 0-concordance from $K_1$ to $K_2$ can be decomposed into a ribbon concordance from $K_1$ to a third knot $K$, followed by a ribbon concordance from $K$ to $K_2$.
\end{proof}


\section{Motivation: Rochlin's Invariant}

In this section we'll describe how Rochlin's invariant can be used to prove a weak version of Theorem \ref{mainthm}. We will begin by reviewing the definition of Rochlin's invariant  for 3-manifolds, and then show how it can be adapted to give invariants of 2-knots. Rochlin's invariant was first applied to 2-knots in the thesis of Ruberman (see \cite{R}), where it is shown to be sensitive to properties like invertibility and amphichirality.

\subsection{Rochlin's invariant for 3-manifolds.}
Rochlin's invariant of a 3-manifold $Y$ with a spin structure $\mathfrak{s}$ is defined as 

\[ \mu(Y,\mathfrak{s}) := \sigma(X) \text{ mod } 16 \]

where $X$ is a spin 4-manifold with a spin structure that restricts to $\mathfrak{s}$ on $Y$, and $\sigma$ represents the signature. See the books of Kirby \cite{K} or Saveliev \cite{Sa} for a thorough description of spin structures and Rochlin's invariant.






\subsection{Rochlin's invarinat of a 2-knot.}

We can now define the Rochlin invariant of an oriented 2-knot to be the Rochlin invariant of any Seifert hypersurface with compatible spin structure. Specifically, given a 2-knot $K$ in $S^4$ with a Seifert hypersurface $Y^{\circ}$ for $K$, define $\mu(K)$ to be $\mu(Y,\mathfrak{s})$, where $\mathfrak{s}$ is induced from the embedding in $S^4$. This definition is essentially that described in \cite{R} and \cite{RS}. 

\begin{prop}\label{roc4}
The definition of $\mu(K)$ does not depend on the choice of Seifert hypersurface.
\end{prop}

\begin{proof}
Suppose $Y_1$ and $Y_2$ are both (capped off) Seifert hypersurfaces of $K$. Surgery on $K$ in $S^4$ give a homology $S^1\times S^3$, denoted X, where $H_3(X)$ is represented by both $Y_1$ and $Y_2$ embedded in $X$. Let $X'$ denote a cyclic cover of $X$ in which $Y$ and $Y'$ are disjoint. By a theorem of Sumners (\cite[Theorem 3]{Sum}), for a prime power cover, $X'$ will have vanishing second Betti number. Moreover, we have that $Y \sqcup -Y'$ bounds a spin 4-manifold $M \subset X'$.

\begin{claim}
$b_2^+(M) = b_2^-(M) = 0$.
\end{claim}
Hence, assuming the claim is true, the signature of $M$ vanishes. Finally, using the additivity of $\mu$ and the fact that $\mu(Y) = -\mu(-Y) \text{ mod } 16$, we have that $\mu(Y_1) - \mu(Y_2) = \mu(Y_1) + \mu(-Y_2) = \mu(Y_1 \sqcup -Y_2) = \sigma(M) = 0$. 

It only remains to prove the claim:
\begin{proof}[Proof of Claim.] 
Suppose $b_2^+(M) \neq 0$. Then $M$ contains a surface with positive self intersection. But since $M$ is a submanifold of $X'$, this says that $X'$ will also contain a surface with positive self intersection. This contradicts the fact that $b_2(X') = 0$. So $b_2^+(M) = 0$, and the proof for $b_2^-(M)$ is identical.
\end{proof}

\end{proof}

For a slightly different proof, see \cite{R}. Using this invariant, we can prove the following simplified version of Theorem \ref{mainthm}

\begin{theorem}\label{roc0}
If $K_1$ and $K_2$ are 2-knots that are 0-concordant, then $\mu(K_1) = \mu(K_2)$. Moreover, there are at least 16 distinct 0-concordance classes of 2-knots.
\end{theorem}

\begin{proof}
By Corolary \ref{seif0}, if $Y_1^{\circ}$ and $Y_2^{\circ}$ are Seifert hypersurfaces for $K_1$ and $K_2$, then $Y_1^{\circ} \# n(S^1\times S^2)$ and $Y_2^{\circ} \# m(S^1\times S^2)$ are both Seifert hypersurfaces for some 2-knot $K$. Therefore (suppressing the spin structures from our notation), $\mu(K_1) = \mu(Y_1) = \mu(Y_1\# nS^1\times S^2) = \mu(Y_2\# nS^1\times S^2) = \mu(Y_2)$. We have used here that $\mu$ is additive under connect sum, and that $\mu(S^1\times S^2) = 0$.

Finally, to demonstrate that all possible values of the Rochlin invariant of a 2-knot are realized, it is enough to find a knot $K$ with $\mu(K)=1$ and use the additivity of $\mu$ under connect sum of knots. For example, the 2-twist spin of the $(2,1)-$torus knot will be a fibered knot with $L(2,1)$ as the fiber (by Proposition \ref{zeeman}), and since the equivariant spin structure on $L(2,1)$ spin bounds a $D^2$ bundle over $S^2$, which has signature 1, the Rochlin invariant of any such knot is 1.

\end{proof}






\begin{remark}
The quandle cohomology invariant (see e.g. \cite{CSS}) can be used to say something about ribbon concordance, but it does not tell us anything about 0-concordance. This is because two different ribbon knots might not be ribbon concordant, and the quandle invariant provides an obstruction. On the other hand, two ribbon knots will always be 0-concordant (by composing ribbon disks).
\end{remark}

\section{Background in Heegaard-Floer homology}

Whereas in the last section we found sixteen distinct 0-concordance classes of 2-knots using Rochlin's invariant, to distinguish more than sixteen 0-concordance classes we will need more refined invariants. These will come from Heegaard-Floer homology. Specifically we will use a variation of the $d$-invariant. The $d$-invariant was first introduced in \cite{OZ} for homology 3-spheres, 3-manifolds with $b_+=1$, and for manifolds with ``standard $HF^{\infty}$.'' These definitions were later extended to ``intermediate'' invariants by Levine and Ruberman in \cite{LR}, to general 3-manifolds (using a slightly different method) in \cite{BG} by Behrens and Golla, and finally placed a much more general context by Levine and Ruberman in \cite{LR2}. Although any of these variations would suffice for our purposes, we will present the version Behrens and Golla which will be the most straightforward for the applications we have in mind.


 What follows are a few of the formal properties of Heegaard-Floer homology used to define these invariants. We denote the field of characteristic 2 by $\mathbb{F}$

For a 3-manifold Y with a $\spinc$ structure $\mathfrak{s}$, the Heegaard-Floer homology is comprised of three abelian groups $HF^+(Y,\mathfrak{s}), HF^-(Y,\mathfrak{s}),$ and $HF^{\infty}(Y,\mathfrak{s})$. These groups have the following additional structure from which we will derive the invariants relevant to this paper. 

Let $\mathcal{T}^{\infty}$, $\mathcal{T}^+$, and $\mathcal{T}^-$ represent the $\mathbb{F}[U]$-modules $\mathbb{F}[U,U^{-1}]$, $\mathbb{F}[U]$, and $\mathbb{F}[U,U^{-1}]/\mathbb{F}[U]$.

\begin{enumerate}
\item When $\mathfrak{s}$ is a torsion $\spinc$ structure, the Heegaard-Floer groups are $\mathbb{Q}$-graded.

\item They have the structure of a $\mathbb{F}[U] \otimes H^1(Y)$ module.

\item A 3-manifold is said to have ``standard $HF^{\infty}$'' if for each torsion $\spinc$ structure $\mathfrak{s}$, $HF^{\infty}(Y, \mathfrak{s}) = \Lambda^*(H_1(Y)/tor)\otimes \mathbb{Z}[U,U^{-1}]$. This is always true if the triple cup-product on $H^1(Y)$ vanishes.

\item If $\mathfrak{s}$ is a torsion $\spinc$ structure on a manifold with standard $HF^{\infty}$, the groups $HF^{\pm}(Y,\mathfrak{s})$ have the structure $\oplus_{2^{b_1(Y)}}\mathcal{T^{\pm}} \oplus \text{f.g. group}$, and $HF^{\infty}(Y,\mathfrak{s})$ has the structure  $\oplus_{2^{b_1(Y)}}\mathcal{T}^{\infty}$.
\item These groups fit into the following long exact triangle:
\[\cdots \rightarrow HF^-(Y,\mathfrak{s}) \rightarrow HF^{\infty}(Y,\mathfrak{s})\rightarrow HF^+(Y,\mathfrak{s})\rightarrow \cdots \]

\end{enumerate}

\subsection{The Behrens-Golla twisted d-Invariants}
The above properties imply that if $Y$ is a homology 3-sphere, then $HF^+(Y) = \mathcal{T}^+ \oplus \{\text{f.g. abelian group}\}$. But they do not tell us anything about the gradings. The original $d$-invariant for a homology 3-sphere was defined to be the smallest grading of any element in the $\mathcal{T}^{+}$ part of $HF^+(Y)$. Alternatively, one can define the $d$-invariant as the smallest grading of any element in the image of $HF^{\infty}(Y)\rightarrow HF^+(Y)$. If, on the other hand, $Y$ is not a homology 3-sphere, one must be more careful, because $HF^{\infty}$ does not necessarily have such a simple form. One way around this, is is to use twisted coefficients to define $\underline{HF}^{\infty}$, $\underline{HF}^{+}$, and $\underline{HF}^{-}$. Then, for a torsion $\spinc$ structure $\mathfrak{s}$ on $Y$, one can show that $\underline{HF}^{\infty}(Y)$ is again standard, and Behrens-Golla define $\underline{d}(Y,\mathfrak{s})$ as the minimal grading in the image of $\underline{HF}^{\infty}(Y,\mathfrak{s})\rightarrow \underline{HF}^+(Y,\mathfrak{s})$.


Those interested in the full details can refer to \cite{BG}. The most important properties of this invariant are summarized in the following statement. 

\begin{prop}\label{dprops}
\begin{enumerate} The invariant $\underline{d}$ satisfies:
\item If $Y$ is a rational homology 3-sphere, then $\underline{d}(Y,\mathfrak{s})$ agrees with the original d-invariant defined by Ozsvath and Szabo, for which $d(Y,\mathfrak{s}) = -d(-Y,\mathfrak{s})$.
\item If $Y = \#_nS^1\times S^2$, then $\underline{d}(Y,\mathfrak{s}_0) = -\frac{n}{2}$ where $\mathfrak{s}_0$ is the trivial $\spinc$ structure.

 
 \item (Additivity) $\underline{d}(Y_1\#Y_2, \mathfrak{s}_1\#\mathfrak{s}_2) = \underline{d}(Y_1, \mathfrak{s}_1) +  \underline{d}(Y_2, \mathfrak{s}_2) $ 
 
\item The inequality \[ c_1^{2}(\mathfrak{s}) + b_2^-(X) \leq 4\underline{d}(Y, \mathfrak{t}) + 2b_1(Y)\] holds, where $X$ is a negative semi-definite 4-manifold bounded by a connected $Y$, and $\mathfrak{s}$ is a $\spinc$ structure on $X$ that restricts to $\mathfrak{t}$ on $Y$. 
 

\end{enumerate}
\end{prop}



\section{Proofs of Theorems 1 and 2}

To prove Theorem 1, the role of Rochlin's invariant in Theorem \ref{roc0} will be replaced by the $\underline{d}$ invariant, and the fact that Rochlin's invariant is invariant under cobordisms with vanishing signature will be replaced with the aforementioned inequality.

In the remainder of this paper, we will assume that all 2-knots are oriented, and that all Seifert hypersurfaces are oriented in such a way to be compatible with their corresponding 2-knots. Furthermore, we will require that all 0-concordances are oriented consistently with the 2-knot orientations. 

\begin{theorem} \label{0conctheorem}
Let $K_1$ and $K_2$ be oriented 0-concordant 2-knots in $S^4$. If they have Seifert manifolds $Y_1^{\circ}$ and $Y_2^{\circ}$, both of which are punctured rational homology 3-spheres, then $d(Y_1,\mathfrak{s}_1) = d(Y_2,\mathfrak{s}_2)$, where $\mathfrak{s}_i$ represents the $\spinc$ structure on $Y_i$ induced from the $\spinc$ structure on $S^4$.
\end{theorem}

\begin{proof}
The proof follows that of Proposition \ref{roc4}, where we replace Rochlin's invariant with the inequality in Proposition \ref{dprops}. Since $K_1$ and $K_2$ are 0-concordant, there is a 2-knot $K$ that has both $\underline{Y}_1^{\circ}= Y_1^{\circ} \# n(S^1\times S^2)$ and $\underline{Y}_2^{\circ} = Y_2^{\circ} \# m(S^1\times S^2)$ as Seifert hypersurfaces by Corollary \ref{seif0}. Surgery along K then gives a homology $S^1\times S^3$, denoted by X, with $\underline{Y}_1$ and $\underline{Y}_2$ representing the same homology class in $H_3(X)$. 

As in Proposition \ref{roc4}, take a high enough prime power cover of $X'$ of $X$ such that $\underline{Y}_1$ and $\underline{Y}_2$ are disjoint in $X'$. Now $\underline{Y}_1 \sqcup -\underline{Y}_2$ bounds a 4-manifold $M \subset X'$ such that $b_2^-(M) = 0$. Finally, construct a manifold $M'$ from $M$ with $\partial M' = \underline{Y}_1 \# -\underline{Y}_2$ by removing an open neighborhood of an arc connecting $\underline{Y}_1$ to $\underline{Y}_2$ in $M$. Moreover $b_2^-(M')$ is still 0. 



Applying the inequality to $M'$ now gives: 
\begin{align*}
c_1^{2}(\mathfrak{s}) + b_2^-(M') & \leq 4\underline{d}(Y_1 \# -Y_2 \#(m+n)S^1\times S^2) + 2b_1(Y_1 \# -Y_2 \#(m+n)S^1\times S^2) \\
0 & \leq 4d(Y_1) - 4d(Y_2) + (4\underline{d}((m+n)S^1\times S^2) + 2b_1((m+n)S^1\times S^2))) \\
0  & \leq 4d(Y_1) - 4d(Y_2) + 0
\end{align*}
So $d(Y_2) \leq d(Y_1)$

The opposite inequality is proved similarly by reversing the orientation of $M'$.

\end{proof}

Using this theorem, it is relatively simple to exhibit many twist spun knots that all lie in different 0-concordance classes.

\begin{proof}[Proof of Theorem 1]
Using the additivity of $\underline{d}$-invariants under connect sum, we only need to find a 2-knot that has a homology 3-sphere Seifert hypersurface with non-trivial $\underline{d}$-invariant. The 5-twist spin on the trefoil has the Poincar\'e homology sphere as a Seifert hypersurface \cite{Z}, which has $\underline{d}$-invariant equal to 2 (see \cite[Section 8]{OZ}). So by taking connected sums of this 2-knot we get an infinite number of oriented 0-concordance classes. Examples of infinite families of prime knots, none of which are 0-concordant arise by taking the 2-twist spin on the $(2,p)$-torus knots, which have the lens space $L(p,1)$ as a Seifert hypersurface, and $d(L(p,1),\mathfrak{s}_0) = \frac{p-1}{4}$.  
\end{proof}

Many more examples from twist spun knots can be computed using, for example, the $d$-invariant calculations for double branched covers in \cite{MO,LO} and the calculations for higher branched covers in \cite{J}. Additional relevant techniques, computations, and examples can be found in Section 5 of \cite{LR2}.

\begin{proof}[Proof of Theorem 2]
Suppose the 2-twist spin of a quasi-alternating knot $K$ is 0-slice. Then by Zeeman's result (Proposition \ref{zeeman}), the double branched cover of $K$, denoted $Y^{\circ}$, is a Seifert hypersurface of this 2-knot, and by Theorem \ref{0conctheorem}, we have that $d(Y, \mathfrak{s}_0) = 0$. (Here note that $\mathfrak{s}_0$ is the $\spinc$ structure induced by the unique spin structure on $Y$, which in turn is equal to the $\spinc$ structure induced from $S^4$). However, the computations of Lisca-Owens in \cite{LO} (following those of Manolescu-Owens in \cite{MO}), show that if $K$ is quasi-alternating, then $2d(Y, \mathfrak{s}_0)$ equals the signature of $K$. 
\end{proof}

\begin{remark}\label{sym}
The strategy using twisted correction terms above can also be employed to obstruct a 2-knot from being amphichiral or invertible. This has been pursued in \cite{LR2} using using a more sophisticated setup that allows one to consider 2-knots with general Seifert hypersurfaces (not just 2-knots that have Seifert hypersurfaces that are a rational homology sphere connect summed with some $S^1\times S^2$'s, like we consider in this paper). 
\end{remark}

\textbf{Acknowledgements:} The author would like to thank Adam Levine for his help in understanding intermediate correction terms, and Stefan Behrens and Marco Golla, both of whom provided extensive comments on an earlier draft of this paper.

\end{document}